\documentclass[english,11pt]{amsart}
\usepackage[T1]{fontenc}
\usepackage[latin9]{inputenc}
\usepackage{babel}
\usepackage{geometry}
\usepackage{amsthm}
\usepackage{amsmath}
\usepackage{cleveref}
\usepackage{relsize}
\usepackage{amssymb}
 \usepackage{amsfonts}
 \usepackage{amscd}
 \usepackage[all]{xy}
\usepackage[dvipsnames]{xcolor}
\usepackage{pdfsync}
\usepackage{url}
\usepackage[sc]{mathpazo}

\makeatletter
  \theoremstyle{remark}
  \theoremstyle{example}
  \newtheorem*{rem*}{\protect\remarkname}
\makeatother

\usepackage{babel}
  \providecommand{\remarkname}{Remark}
  \providecommand{\lemmaname}{Lemma}

\theoremstyle{plain}
\newtheorem{thm}{\protect\theoremname}[section]  

  \theoremstyle{definition}
  \newtheorem{example}[thm]{\protect\examplename}
  \theoremstyle{definition}
  \newtheorem{defn}[thm]{\protect\definitionname}
  \theoremstyle{remark}
  \newtheorem{rem}[thm]{\protect\remarkname}
  \theoremstyle{lemma}
  \newtheorem{lemma}[thm]{\protect\lemmaname}
    \theoremstyle{prop}
  \newtheorem{prop}[thm]{\protect\propositionname}
    \theoremstyle{conj}
  \newtheorem{conj}[thm]{\protect\conjname}
    \theoremstyle{question}
  \newtheorem{question}[thm]{\protect\questionname}
    \theoremstyle{cor}
  \newtheorem{cor}[thm]{\protect\corname}

\makeatother

\usepackage{babel}
  \providecommand{\definitionname}{Definition}
  \providecommand{\examplename}{Example}
\providecommand{\theoremname}{Theorem}
\providecommand{\propositionname}{Proposition}
\providecommand{\conjname}{Conjecture}
\providecommand{\questionname}{Question}
\providecommand{\corname}{Corollary}

\newcommand{\cX}{X}

\newcommand{\Xmn}{X_{mn}}

\newcommand{\Ymn}{Y_{mn}}
\newcommand{\oYmn}{Y_{mn}^o}

\newcommand{\mfX}{\mathcal X}
\newcommand{\mfY}{\mathcal Y}

\newcommand{\cG}{\mathcal G}

\newcommand{\cU}{\mathcal U}
\newcommand{\cF}{\mathcal F}
\newcommand{\Lams}{\Lambda}

\newcommand{\vX}{X}
\newcommand{\oX}{X^o}

\newcommand{\oZ}{Z^o}

\newcommand{\aaCount}{20}

\newcommand{\cc}{\mathbb C}
\newcommand{\zz}{\mathbb Z}
\newcommand{\pp}{\mathbb P}
\newcommand{\bfc}{{\bf c}}

\DeclareMathOperator{\gdeg}{gdeg}
\DeclareMathOperator{\MLdegY}{MLdeg^o}
\DeclareMathOperator{\MLdegX}{MLdeg}

\DeclareMathOperator{\rank}{rank}

\begin{document}

\title{\vspace{-4ex}  The maximum likelihood degree of mixtures of independence models}

\author[Rodriguez]{Jose Israel Rodriguez}
\address{
Jose Israel Rodriguez\\
The University of Chicago\\
Dept. of Statistics\\
5734 S. University Ave.
Chicago, IL 60637}
\email{joisro@Uchicago.edu}
\urladdr{http://home.uchicago.edu/~joisro}

\author[Wang]{Botong Wang}
\address{
Botong Wang\\
University of Wisconsin-Madison\\
 Department of Mathematics\\
Van Vleck Hall, 480 Lincoln Drive, 
Madison, WI}
\email{wang@math.wisc.edu}
\urladdr{http://www.math.wisc.edu/~wang/}

\maketitle\vspace{-4ex} 
\begin{abstract}
\noindent
The maximum likelihood degree (ML degree) measures the algebraic complexity of a fundamental optimization problem in statistics: maximum likelihood estimation. 
In this problem, one maximizes the likelihood function over a statistical model. 
The ML degree of a model is an upper bound to the number of local extrema of the likelihood function and can be expressed as a weighted sum of Euler characteristics. 	
The independence model (i.e. rank one matrices over the probability simplex) is well known to have an ML degree of one, meaning there is a unique local maximum of the likelihood function. 
However, for mixtures of independence models (i.e. rank two matrices over the probability simplex),  
it was an open question as to how the ML degree behaved.
In this paper, we use Euler characteristics to  prove an outstanding conjecture by Hauenstein, the first author, and Sturmfels;
we give recursions and closed form expressions for the ML degree of mixtures of independence models.

\end{abstract}

\section{Introduction}
{
Maximum likelihood estimation  is a fundamental computational task in statistics. 
 A typical problem encountered in its applications is the occurrence of multiple local maxima.
To be certain that a global maximum of the likelihood function has been achieved,
one locates all solutions to a system of polynomial equations called \textit{likelihood equations}; 
every local maxima is a solution to these equations. 
The number of solutions to these equations is called the maximum likelihood degree (ML degree) of the model. 
This degree was 
 introduced in \cite{CHKS06,HKS05} and 
gives a measure of the complexity to the global optimization problem, as it bounds the number of local maxima.

The maximum likelihood degree has been studied in many contexts.
Some of these contexts include Gaussian graphical models \cite{Uhler12}, variance component models \cite{GDP12}, and in missing data \cite{HS09}. 
In this manuscript, we work in the context of discrete random variables (for a recent survey in this context, see \cite{HS2014}).
In our main results, we provide closed form expressions 
for ML degrees of mixtures of independence models, which are sets of joint probability distributions for two random variables. 
This answers the outstanding Conjecture 4.1 in \cite{HRS14}.

\subsection{Algebraic statistics preliminaries}
We consider a model for two discrete random variables, having 
$m$ and $n$ states respectively.  
A joint probability distribution for two such random variables is written as an 
$m \times n$-matrix:
{\smaller
\begin{equation}\label{theP}
P=\left[
\begin{array}{cccc}
p_{11} & p_{12}&\cdots & p_{1n}\\
p_{21} & p_{22}&\cdots& p_{2n}\\
\vdots & \vdots & \ddots&\vdots \\
p_{m1} & p_{m2} &\cdots &p_{mn}
\end{array}
\right].
\end{equation}}%
The $(i,j)$th entry $p_{ij}$ represents the probability that the
first variable is in state $i$
and the second variable is in state~$j$. 
By a statistical model, we mean a  subset
${\mathcal M}$ of the 
probability simplex $\triangle_{mn-1}$
of all such matrices $P$.  
The $d-1$ dimensional probability simplex is defined as 
$\triangle_{d-1}~:=~\left\{(p_1,p_2,\dots,p_d)\in\mathbb{R}^d_{\geq 0}\text{ such that } \sum p_i=1   \right\}$.~

If i.i.d. samples are drawn from some distribution $P$,
then the \textit{data} is summarized by the following matrix \eqref{theU}.
The entries of $u$ are non-negative integers where $u_{ij}$ is the number of samples drawn with state
$(i,j)$:
{\smaller
\begin{equation}\label{theU}
u=\left[
\begin{array}{cccc}
u_{11} & u_{12}&\cdots & u_{1n}\\
u_{21} & u_{22}&\cdots& u_{2n}\\
\vdots & \vdots & \ddots&\vdots \\
u_{m1} & u_{m2} &\cdots &u_{mn}
\end{array}
\right].
\end{equation}
}

The likelihood function corresponding to the data matrix $u$ is given by 
\begin{equation}\label{likelihoodFunction}
\ell_u(p):=p_{11}^{u_{11}}p_{12}^{u_{12}}\cdots p_{mn}^{u_{mn}}.
\end{equation}
Maximum likelihood estimation is an optimization problem for the likelihood function. This problem consists of determining, for fixed $u$, the argmax of $\ell_u (p)$ on a statistical model ${\mathcal M}$. 
The optimal solution is called the maximum likelihood estimate (mle) and is used to measure the true probability distribution. 
For the models we consider, the mle is a solution to the likelihood equations. 
In other words, by solving the likelihood equations, we solve the maximum likelihood estimation problem. 
Since the ML degree is the number of solutions to the likelihood equations,  it gives a measure on the difficulty of the problem. 

The model ${\mathcal M}_{mn}$ in $\triangle_{mn-1}$ is said to be the \textit{mixture of independence models} and is 
 defined to be the image of the following map:
\begin{equation}
\begin{array}{ccc}
\left(\triangle_{m-1}\times\triangle_{n-1}\right)\times\left(\triangle_{m-1}\times\triangle_{n-1}\right)\times\triangle_{1} & \to & {\mathcal M}_{mn}\\
\left(R_{1},B_{1},R_{2},B_{2},C\right) & \mapsto & c_{1}R_{1}B_{1}^{T}+c_{2}R_{2}B_{2}^{T}
\end{array},
\end{equation}
where $R_i,B_j,C:=[c_1,c_2]^T$ are $m\times 1$ matrices, $n\times 1$, and $2\times 1$ matrices respectively with positive entries that sum to one.
The Zariski closure of ${\mathcal M}_{mn}$ yields the variety of rank at most $2$ matrices over the complex numbers. 
We will determine the ML degree of the models ${\mathcal M}_{mn}$ by studying the topology of the Zariski closure.
Prior to our work, the ML degrees of these models were only known for small values of $m$ and $n$. 
In \cite{HRS14}, the following table of ML degrees of  ${\mathcal M}_{mn}$ were computed:
$$
\begin{array}{ccccccccccc}
 n= & 3 & 4 & 5 & 6 & 7 \\
m=3: & 10 & 26 & { 58} & { 122} & { 250}\\
m=4: & 26 & { 191} & { 843} & { 3119} &  \\
m=5: & 58 & { 843} & { 6776} &  & & . 
\end{array}$$
Mixtures of independence  models appear in many places in science, statistics, and mathematics. 
In computational biology, the case where $(m,n)$ equals $(4,4)$ is discussed in Example 1.3 of \cite{PS05}, and the data $u$ consists of a pair of DNA sequences of length $(u_{11}+u_{12}+\cdots+u_{mn})$. 
The $m$ and $n$ equal four because 
DNA molecules are composed of four nucleotides.
Another interesting case for computational biology is when $(m,n)=(\aaCount,\aaCount)$ because there are $\aaCount$ essential  amino acids [Remark \ref{AminoAcidRemark}].

Our first main result, Theorem \ref{hrsWin}, proves a formula for the first row of the table: 
$$
\texttt{MLdegree}{\mathcal M}_{3n}=2^{n+1}-6\texttt{ for }n\geq 3.
$$
Our second main result, is a recursive expression for the ML degree of mixtures of independence models. 
As a consequence, we are able to calculate a
closed form expression  for each row of the above table of ML degrees [Corollary \ref{compExpressions}]. 

Our techniques relate  ML degrees to Euler characteristics. Previously,
 Huh has shown that the ML degree of a smooth algebraic statistical models ${\mathcal M}$ with Zariski closure $X$ equals  the signed topological Euler characteristic of an open subvariety $X^o$, where 
$X^o$ is the set of points of $X$ with nonzero coordinates and coordinate sums  \cite{Huh13}. 
More recent work of Budur and the second author shows that the ML degree of a singular model is a stratified topological invariant. 
In \cite{BW14}, they show that the Euler characteristic of $X^o$ is a sum of ML degrees weighted by Euler obstructions.
These Euler obstructions, in a sense, measure the multiplicity of the singular~locus.

We conclude this introduction with illustrating examples to set notation and definitions.
}

\subsection{Defining the maximum likelihood degree}
 We will use two notions of maximum likelihood degree. 
 The first notion is from a computational algebraic geometry perspective, where we define the maximum likelihood degree  for a projective variety. 
 When this projective variety is contained in a hyperplane, the maximum likelihood degree has an interpretation related to statistics.
The second notion is from a topological perspective, where we define the maximum likelihood degree for a very affine variety, a subvariety of an algebraic torus $(\cc ^*)^n$.  

To $\mathbb{P}^{n+1}$,
 we associate the coordinates $p_0,p_1,\dots,p_n,$ and $p_s$ (were $s$ stands for sum). 
Consider the \textit{distinguished hyperplane} in $\mathbb{P}^{n+1}$ defined by $p_0+\cdots+p_n-p_s=0$ ($p_s$ is the sum of the other coordinates).

Let $\cX$ be a generically reduced variety contained in the distinguished hyperplane  of $\mathbb{P}^{n+1}$ not contained in any coordinate hyperplane. 
We will be interested in the critical points of the \textit{likelihood function} 
$$\ell_u(p):=p_0^{u_0}p_1^{u_1}\cdots p_n^{u_n}p_s^{u_s}$$
where $u_s:=-u_0-\cdots -u_n$ and $u_0,\dots,u_n\in\mathbb{C}$. The likelihood function has the nice property that, up to scaling,
its gradient is a rational function $\nabla \ell_u(p):=[\frac{u_0}{p_0}:\frac{u_1}{p_1}:\dots:\frac{u_n}{p_n}:\frac{u_s}{p_s}]$.
 
\begin{defn} Let $u$ be fixed. 
A point $p\in\cX$   is said to be a \textit{critical point of the likelihood function on} $\cX$ if  
$p$ is a regular point of $\cX$,
 each coordinate of $p$ is nonzero, and 
 the gradient $\nabla \ell_u(p)$ at $p$  is orthogonal to the tangent space of $\cX$ at $p$.
\end{defn}
\begin{example}\label{exampleDet22}
Let $\vX$ of $\pp^4$ be defined by $p_0+p_1+p_2+p_3-p_s$ and $p_0p_3-p_1p_2$.
For $[u_0:u_1:u_2:u_3:u_s]=[2:8:5:10:-25]$ there is a unique critical point  for $\ell_u(p)$ on $\vX$.  This point is $[p_0:p_1:p_2:p_3:p_s]=[70:180:105:270:-625]$.
Whenever the coordinates are nonzero, there is a unique critical point 
$[(u_0+u_1)(u_0+u_2):(u_0+u_1)(u_1+u_3):
(u_2+u_3)(u_0+u_2):(u_2+u_3)(u_1+u_3):
-(u_0+u_1+u_2+u_3)^2]$.
\end{example}

\begin{defn}
The \textit{maximum likelihood degree of} $\cX$  is defined to be the number of critical points of the  likelihood function on $\cX$ for general $u_0,\dots,u_n$. The maximum likelihood degree of $X$ is denoted $\MLdegX(\cX)$.
\end{defn}

We say $u^*$ in $\mathbb{C}^{n+1}$ is general, whenever there exists a dense Zariski open set  $\cU$ for which the number of critical points of $\ell_u(p)$ is constant and $u^*\in\cU$.
In Example \ref{exampleDet22}, the Zariski open set $\cU$ is the complement of the variety defined by 
$(u_0 + u_1)(u_0 + u_2)(u_1 + u_3)(u_2 + u_3)u_s=0$.
Determining this Zariski open set explicitly is often quite difficult, but it is not necessary when using reliable 
 probabilistic algorithms  to compute ML degrees as done in \ref{HKS05} for example. 
However, with our results, we compute ML degrees using Euler characteristics and topological arguments.

\subsection{Using Euler characteristics}
First, we recall that the Euler characteristic of a topological space $X$ is defined as
$$\chi(X)=\sum_{i\geq 0}(-1)^i\dim H^i(X, \mathbb{Q}),$$
where the $H^i$ are the singular cohomology groups.  

In the definition of maximum likelihood degree of a projective variety $\vX$, a critical point $p\in \vX$ must have nonzero coordinates. This means all critical points of the likelihood function are contained in the underlying very affine variety of 
$$\vX^o:=\vX\backslash\{\text{coordinate hyperplanes} \}.$$
In fact, the ML degree is directly related to the Euler characteristic of smooth $\vX^o$.

\begin{thm}[\cite{Huh13}]
Suppose $\vX$ is a smooth projective variety of $\mathbb{P}^{n+1}$. 
Then, 
\begin{equation}\label{smoothCase}
\chi(\vX^o)=(-1)^{\dim \vX^o}\MLdegX \vX.
\end{equation}
\end{thm}

The next example will show how to determine the signed Euler characteristic of a very affine variety $Y$. 
Recall that the Euler characteristic is a homotopy invariant and satisfies the following properties. 
The Euler characteristic is \textit{additive} for algebraic varieties. More precisely, $\chi(X)=\chi(Z)+\chi(X\setminus Z)$, where $Z$ is a closed subvariety of $X$ (see \cite[Section 4.5]{FultonToric} for example).
The \textit{product property} says $\chi(M\times N)=\chi(M)\cdot\chi(N)$.
More generally, the \textit{fibration property} says that if $E\to B$ is a fibration with fiber $F$ then 
$\chi(E)=\chi(F)\cdot \chi(B)$ (see \cite[Section 9.3]{Spanier} for example). 

\begin{example}\label{topologicalExample}
Consider $\vX$ from Example \ref{exampleDet22}. 
The variety $\vX$ has the parameterization shown below 
\[
\begin{array}{ccl} 
\mathbb{P}^{1}\times\mathbb{P}^{1} & \to & \vX\\
\left[x_{0}:x_{1}\right]\times\left[y_{0}:y_{1}\right] & \mapsto & 
\left[x_{0}y_{0}:x_{0}y_{1}:x_{1}y_{0}:x_{1}y_{1}:
x_0y_0+x_0y_1+x_1y_0+x_1y_1)\right].
\end{array}
\]
Let $\oX$ be the underlying very affine variety of $\vX$ and consider  
$\mathcal O:=\pp^{1}\backslash \left\{[0:1],[1:0],[1:-1] \right\}$
  the projective space with $3$ points removed. 
The very affine variety  
 $\oX$  has a parameterization: 
 \[
\begin{array}{ccl}
\mathcal O\times\mathcal O & \to & \oX\\
\left[x_{0}:x_{1}\right]\times\left[y_{0}:y_{1}\right] & \mapsto & 
\left[x_{0}y_{0}:x_{0}y_{1}:x_{1}y_{0}:x_{1}y_{1},
{(x_0+x_1)(y_0+y_1)}\right].
\end{array}
\]
Since $\chi(\pp^1)=2$, after removing 3 points, we have $\chi(\mathcal O)=-1$.
By the product property $\chi(\mathcal O\times \mathcal O)=1$, and hence $\chi(\oX)=1$.
Because $\oX$ is smooth,  by Huh's result, we conclude the ML degree of $\oX$ is $1$ as well.

We call $\vX$ the  variety of $2\times 2$   matrices with rank $1$ and restricting to the probability simplex yields the independence model. 
This example generalizes  by considering the map $\mathbb{P}^{m-1}\times \mathbb{P}^{n-1}\to \vX$ given by 
$([x_0:\dots:x_{m-1}],[y_0:\dots:y_{n-1}])\to [x_0y_0:\dots:x_{m-1}y_{n-1}:\sum_{i,j}x_iy_j]$. In this case, $\vX$ is  the variety of $m\times n$ rank $1$ matrices, and a similar computation shows that the ML degree is $1$ in these cases as well (see Example 1 of \cite{HRS14}). 
\end{example}


\section{Whitney stratification, Gaussian degree, and Euler obstruction}
As we have just seen, for a smooth very affine variety, the ML degree is equal to the Euler characteristic up to a sign. This is not always the case, when the very affine variety is not smooth. 
When the variety is singular, the ML degree is related to the Euler characteristic in a subtle way involving Whitney stratifications. 
This subtlety is explained in Corollary \ref{MLdegformula} by considering a weighted sum of the ML degrees of each strata.
Here, we give a brief introduction to the topological notions of Whitney stratification and Euler obstruction.

\subsection{Whitney stratification}

Many differential geometric notions do not behave well when a variety has singularities, for instance, tangent bundles and Poincare duality. 
This  situation is addressed  by stratifying  the singular variety into finitely many pieces, such that along each piece the variety is close to a smooth variety. 
A naive way to stratify a variety $X$ is taking the regular locus $X_{\text{reg}}$ as the first stratum, and then take the regular part of the singular locus of $X$, i.e., $(X_{\text{sing}})_{\text{reg}}$, and repeat this procedure. 
This naive stratification does not always reflect the singular behavior of a variety as seen in the Whitney umbrella.
Whitney introduced conditions (now called  \textit{Whitney regular}) on a stratification, where many differential geometric results can be generalized to singular varieties.  
 A \textit{Whitney stratification} satisfies the following technical condition of Whitney regular. The definition below is from \cite[Page 37]{GM83}. See also \cite[E3.7]{HTT08}. 
\begin{defn}
Let $X$ be an analytic subvariety of a complex manifold $M$. Let $X=\bigsqcup_{j\in A}S_j$ be a stratification of $X$ into finitely many locally closed submanifolds of $M$. This stratification is called a Whitney stratification if all the pairs $(S_\alpha, S_\beta)$ with $S_\alpha\subset \bar{S}_\beta$ are Whitney regular, which means the following. 
\begin{quote}
Suppose $x_i\subset S_\beta$ is a sequence of points converging to some point $y\in S_\alpha$. Suppose $y_i\subset S_\beta$ also converge to $y$, and suppose that (with respect to some local coordinate system on $M$) the secant lines $l_i=\overline{x_i y_i}$ converge to some limiting line $l$, and the tangent planes $T_{x_i}S_\beta$ converge to some limiting plane $\tau$. Then $l\subset \tau$. 
\end{quote}

\end{defn}

\begin{example}\label{WSexample}Here are some instances of Whitney stratifications. 
\begin{enumerate}
\item \cite[Theorem 19.2]{Whitney65} Every complex algebraic and analytic variety admits a Whitney stratification. 
\item If $X$ is a smooth variety, then the trivial stratification $X$ itself is a Whitney stratification. 
\item If $X$ has isolated singularities at $P_1, \ldots, P_l$, then the stratification $\{P_1\}\sqcup \ldots\sqcup \{P_l\}\sqcup X\setminus \{P_1, \ldots, P_l\}$ is a Whitney stratification. 
\item\label{openset}Since Whitney regular is a local condition, if $\bigsqcup_{j\in A}S_j$ is a whitney stratification of an algebraic variety $X$, and if $U$ is an open subset of $X$, then $\bigsqcup_{j\in A}(S_j\cap U)$ is a Whitney stratification of $U$. 
\item\label{wsSmooth}  If $X$ is a smooth algebraic variety, and $Z\subset X$ is a closed smooth subvariety, then the pair $\{Z, X\setminus Z\}$ is also a Whitney stratification of $X$.  
\item If $\bigsqcup_{i\in I}S_i$ is a Whitney stratification of $X$ and $\bigsqcup_{j\in J}T_j$ is a Whitney stratification of $Y$, then $\bigsqcup_{i\in I, j\in J}S_i\times T_j$ is a Whitney stratification of $X\times Y$. 
\end{enumerate}
\end{example}
 

  
\begin{prop}
Let $X$ be an algebraic variety. Suppose an algebraic group $\cG$ acts on $X$ with only finitely many orbits $O_1, \ldots, O_l$. 
Then $O_1\sqcup \ldots\sqcup O_l$ form a Whitney stratification of $X$. 
\end{prop}
\begin{proof}
First of all, we need to show that each $O_i$ is a locally closed smooth subvariety of $X$. Let $f: \cG\times X\to X$ be the group action map. Then each orbit is of the form $f(\cG\times \{point\})$, and hence is a constructible subset of $X$. On the other hand, since $\cG$ acts transitively on each $O_i$, as constructible subsets of $X$, $O_i$ are isomorphic to smooth varieties. This implies that if $\dim O_i=\dim X$, then $O_i$ is open in $X$. Therefore, $X\setminus \bigcup_{\dim O_i=\dim X}O_i$ is a closed (possibly reducible) subvariety of $X$. Now, $\cG$ acts on $X\setminus \bigcup_{\dim O_i=\dim X}O_i$ with finite orbits. Thus, we can use induction to conclude that each $O_i$ is a locally closed smooth subvariety of $X$. 

Without loss of generality, we assume that $O_1\subset \bar{O}_2$, and we need to show that the pair $O_1, O_2$ is Whitney regular. Moreover, by replacing $X$ by $\bar{O}_2$, we can assume that $O_2$ is open in $X$. By \cite[Lemma 19.3]{Whitney65}, there exists an open subvariety $U$ of $O_1$ such that the pair $U, O_2$ is Whitney regular. Given any $\tau\in \cG$, the action $\tau: \bar{O}_2\to \bar{O}_2$ is an algebraic map, which preserves $O_2$. Thus, the pair $\tau(U), O_2$ is also Whitney regular. Since $\cG$ acts transitively on $O_1$, 
$$\bigcup_{\tau\in \cG}\tau(U)=O_1.$$
Therefore, the pair $O_1, O_2$ is Whitney regular. 
\end{proof}
Consider $\mathbb{P}^{mn-1}=\{[a_{ij}]_{1\leq i\leq m, 1\leq j\leq n}\}$ as the projective space of $m\times n$ matrices. The left $Gl(m, \cc)$ action and the right $Gl(n, \cc)$ action on $\mathbb{P}^{mn-1}$ both preserve the rank of the matrices. 
Moreover, the orbits of the total action of $Gl(m, \cc)\times Gl(n, \cc)$ are the matrices with fixed rank. Therefore, the stratification by rank gives a Whitney stratification of the subvariety $X_r:=\{[a_{ij}]|\rank\left([a_{ij}]\right)\leq r\}$. In particular, we have the following corollary. 
\begin{cor}\label{WhitneyS}
Define $X:=\{[a_{ij}]|\rank\left([a_{ij}]\right)\leq 2\}$ and $Z:=\{[a_{ij}]|\rank\left([a_{ij}]\right)\leq 1\}$. Then the stratification $Z, X\setminus Z$ is a Whitney stratification of $X$. 
\end{cor}


We use Whitney stratifications in  Corollary \ref{MLdegformula}. We will see, up to a sign, the ML degree of a singular variety is equal to the Euler characteristic  with some correction terms. 
The correction terms are linear combinations of the ML degree of smaller dimensional strata of the Whitney stratification, whose coefficients turn out to be the Euler obstructions. 


\subsection{Gaussian degrees}
We have defined the notion of maximum likelihood degree of a projective variety. Sometimes, it is more convenient to restrict the projective variety to some affine torus and consider the notion of maximum likelihood degree of a subvariety of the affine torus. 


We call a closed irreducible subvariety of an affine torus $(\cc^*)^n$ a \textit{very affine variety}. Denote the coordinates of $(\cc^*)^n$ by $z_1, z_2, \ldots, z_n$. The likelihood functions in the affine torus $(\cc^*)^n$ are of the form 
$$l_u=z_1^{u_1}z_2^{u_2}\cdots z_n^{u_n}.$$
\begin{defn}
Let $\mfX\subset (\cc^*)^n$ be a very affine variety. Define the maximum likelihood degree of $\mfX$, denoted by $\MLdegY(\mfX)$, to be the number of critical points of a likelihood function $l_u$ for general $u_1, u_2, \ldots, u_n$. 
\end{defn}
Fix an embedding of $\pp^n\to \pp^{n+1}$ by $[p_0: p_1: \ldots: p_n]\mapsto [p_0: p_1: \ldots: p_n: p_0+p_1+\cdots+p_n]$. Given a projective variety $X\subset \pp^{n}$, we can consider it as a subvariety of $\pp^{n+1}$ by the embedding we defined above. Then as a subvariety of $\pp^{n+1}$, $X$ is contained in the hyperplane $p_0+p_1+\cdots+p_n-p_s=0$. 

Consider $(\cc^*)^{n+1}$ as an open subvariety of $\pp^{n+1}$, given by the open embedding 
$$(z_0, z_1, \ldots, z_n)\mapsto [z_0: z_1:\ldots: z_n: 1].$$
Now, for the projective variety $X\subset \pp^n$, we can embed $X$ into $\pp^{n+1}$ as described above, and then take the intersection with $(\cc^*)^{n+1}$. Thus, we obtain a very affine variety, which we denote by $\oX$. 

\begin{lemma}\label{projectiveveryaffine}
The ML degree of $X$ as a projective variety is equal to the ML degree of $\oX$ as a very affine variety, i.e. 
\begin{equation}
\MLdegX(X)=\MLdegY(\oX). 
\end{equation}
\end{lemma}
\begin{proof}
Fix general $u_0, u_1, \ldots, u_n\in \cc$. The ML degree of $X$ is defined to be the number of  critical points of the likelihood function $(p_0/p_s)^{u_0}(p_1/p_s)^{u_1}\cdots(p_n/p_s)^{u_n}$. The ML degree of $\oX$ is defined to be the number of critical points of $z_0^{u_0}z_1^{u_1}\cdots z_n^{u_n}$. The two functions are equal on $\oX$. Therefore, they have the same number of critical points. 
\end{proof}

For the rest of this section, by maximum likelihood degree we always mean maximum likelihood degree of very affine varieties. 

As observed in \cite{BW15}, the maximum likelihood degree is equal to the Gaussian degree defined by Franecki and Kapranov \cite{FK00}. The main theorem of \cite{FK00} relates the Gaussian degree with Euler characteristics. In this section, we will review their main result together with the explicit formula from \cite{EM99} to compute characteristic cycles. 

First, we follow the notation in \cite{BW15}. Fix a positive integer $n$ for the dimension of the ambient space of the very affine variety. 
Denote the affine torus $(\cc^*)^n$ by $G$ and denote its Lie algebra by $\mathfrak{g}$. Let $T^*G$ be the cotangent bundle of $G$.  Thus, $T^*G$ is a $2n$ dimensionally ambient space with a canonical symplectic structure.
 For any $\gamma\in \mathfrak{g}^*$, let $\Omega_\gamma\subset T^*G$ be the graph of the corresponding left invariant 1-form on $G$. 

Suppose $\Delta\subset T^*G$ is a Lagrangian subvariety of $T^*G$. For a generic $\gamma\in \mathfrak{g}^*$, the intersection $\Delta\cap \Omega_\gamma$ is transverse and consists of finitely many points. 
The number of points in $\Delta\cap \Omega_\gamma$ is constant when $\gamma$ is contained in a nonempty Zariski open subset of $\mathfrak{g}^*$. This number is called the Gaussian degree of $\Delta$, and denoted by $\gdeg(\Delta)$. 

Let $\mfX\subset G$ be an irreducible closed subvariety. 
Denote the conormal bundle of $\mfX_{\text{reg}}$ in $G$ by $T^*_{\mfX_{\text{reg}}}G$, and denote its closure in $T^*G$ by $T_\mfX^*G$. Then $T_\mfX^*G$ is an irreducible conic Lagrangian subvariety of $T^*G$. 
In the language of  \cite{HS2014}, the $T^*_{\mfX}$ is the likelihood correspondence.
Given any $\gamma\in \mathfrak{g}^*$, the left invariant 1-form corresponding to $\gamma$ degenerates at a point $P\in \mfX$ if and only if $T_\mfX^*G\cap \Omega_\gamma$ contains a point in $T_P^*G$. 
Thus, we have the following lemma. 

\begin{lemma}[\cite{BW15}]
Let $\mfX\subset G$  be an irreducible closed subvariety 
where $G=(\mathbb{C}^{*})^n$. Then,
\begin{equation}\label{useless1}
\MLdegY(\mfX)=\gdeg(T^*_\mfX G).
\end{equation}
\end{lemma}

Let $\cF$ be a bounded constructible complex on $G$ and let $CC(\cF)$ be its characteristic cycle. Then $CC(\cF)=\sum_{1\leq j\leq l} n_j\cdot[\Delta_j]$ is a $\zz$-linear combination of irreducible conic Lagrangian subvarieties in the cotangent bundle $T^*G$. The Gaussian degree and Euler characteristic are related by the following theorem. 

\begin{thm}[\cite{FK00}]
With the notation above, we have
\begin{equation}\label{useless2}
\chi(G, \cF)=\sum_{1\leq j\leq l} n_j\cdot \gdeg(\Delta_j).
\end{equation}
\end{thm}

\subsection{Euler obstructions}
The \textit{Euler obstructions} are defined to be the coefficients of some characteristic cycle decomposition (see equation (\ref{useless3})), and it is a theorem of Kashiwara that they can be computed as the Euler characteristic of a complex link (see Theorem \ref{Kashiwara}). 

Let $\bigsqcup_{j=1}^k S_j$ be a Whitney stratification of the very affine variety $\mfX$ with $S_1=Y_{\text{reg}}$. 
Let $e_{j1}$ be the Euler obstruction of the pair $(S_j, \mfX_{\text{reg}})$, which measures the singular behavior of $\mfX$ along $S_j$. More precisely, $e_{j1}$ are defined such that the following equality holds (See \cite[1.1]{EM99} for more details). 
\begin{equation}\label{useless3}
CC(\cc_{S_1})=e_{11}[T^*_{\mfX_{\text{reg}}}G]+\sum_{2\leq j\leq k}e_{j1}[T^*_{S_j}G].
\end{equation}
For example, $e_{11}=(-1)^{\dim \mfX}$. 

Since $\chi(S_1)=\chi(G, \cc_{S_1})$, combining \cref{useless1,useless2,useless3}, we have the following corollary that expresses $\chi(S_1)$  as a weighted sum of ML degrees.

\begin{cor}\label{MLdegformula}
Let $\bigsqcup_{j=1}^k S_j$  
be a Whitney stratification of $\mfX$ with $S_1=\mfX_{\text{reg}}$. Then,
$$\chi(\mfX_{\text{reg}})=e_{11}\MLdegY(\mfX)+\sum_{2\leq j\leq k}e_{j1}\MLdegY(\bar{S}_j),$$
%
where $\bar{S}_j$ denotes the closure of $S_j$ in $\mfX$ 
 and $e_{j1}$ is the Euler obstruction of the pair $(S_j, S_1)$.
\end{cor}
%


Even though the abstract definition of Euler obstruction uses characteristic cycles, there is a concrete topological formula computing Euler obstructions due to Kashiwara. Here we recall that the Euler characteristic with compact support of a topological space $M$ is defined by:
$$\chi_c(M)=\sum_{i\geq 0}\dim H^i_c(M, \mathbb{Q})$$
where $H^i_c$ is the $i$-th cohomology with compact support. When $M$ is a $d$-dimensional (possibly non-compact) orientable manifold, the Poincar\'e duality of $M$ implies that
$$\dim H^i(M, \mathbb{Q})=\dim H^{d-i}_c(M, \mathbb{Q})$$
and hence $\chi_c(M)=(-1)^d\chi(M)$. 


\begin{thm}[Kashiwara\footnote{Here the formula is in the form of \cite[Theorem 1.1]{EM99}, see also \cite[Page 100]{Dim04} and \cite[8.1]{Gin86}.}]\label{Kashiwara}
Under the above notations, fix a point $z\in S_j$, and an embedding of a neighborhood of $z$ in $X$ to an affine space $W$. Let $V$ be an affine subspace of $W$ with $\dim V+\dim S_j=\dim W$, which intersects $S_j$ transversally at $z$. 

Then the Euler obstruction can be computed using Euler characteristic with compact support:
\begin{equation}\label{Eulerformula}
 e_{j1}
=(-1)^{\dim Z+1}\chi_c\left(B\cap Y_{\text{reg}}\cap \phi^{-1}(\epsilon)\right)
\end{equation}
where $B$ is a ball of radius $\delta$ in $W$ centered at $z$, $\phi$ is a general linear function on $V$ vanishing at $z$ and $0< |\epsilon|\ll \delta\ll1$. 
\end{thm}

The affine subspace $V$ is called a "normal slice" of $S_j$, and the intersection $B\cap Y_{\text{reg}}\cap \phi^{-1}(\epsilon)$ is called a "complex link" of the pair $S_j, Y_{\text{reg}}$. By the formula above, every Euler obstruction we consider is computable.


\begin{example}\label{threeByThreeCase}
Consider the variety of rank at most 2 matrices $X$ in $\mathbb{P}^9$
defined by 
$$\det [p_{ij}]_{3\times 3}=p_{11}+p_{12}+\cdots+p_{33}-p_s=0.$$
Denote the very affine  subvariety of $X$ by $\oX$, that is, 
$$\oX:=\{p\in X|p_{ij}\neq 0 \text{ for all } i,j\text{ and }p_s\neq 0\}.$$
The Whitney stratification of $\oX$ consists of $S_1=\oX_{\text{reg}}$, and $S_2$, the singular point of $\oX$, which are the rank $1$ matrices.
By Corollary \ref{MLdegformula}, we have 
\begin{equation}\label{ML12}
\chi(\oX_{\text{reg}})=e_{11}\MLdegY(\oX)+e_{21}\MLdegY(\bar{S}_2).
\end{equation}

With this equation, we determine the ML degree of $X$.
The rank 1 matrices are known to have ML degree one, so  $\MLdegY(\bar{S}_2)=1$. 
The Euler obstruction $e_{21}$ is equal to the Euler characteristic of the complex link as in Theorem \ref{Kashiwara}. 
In this case, the complex link turns out to be homeomorphic to a vector bundle over $\pp^1$ (see Lemma \ref{complexlink} for more details). The sign in front of the Euler characteristic is negative, and hence $e_{21}=-2$. 

The Euler obstruction $e_{11}$ is much easier to determine because this always equals $(-1)^{\dim X}$. So here $e_{11}=-1$. 
In Subsection \ref{calcEC}, we will calculate the Euler characteristic of $\chi(\oX_{\text{reg}})$. In fact, $\chi(\oX_{\text{reg}})=-12$.
Therefore, (\ref{ML12}) implies that $\MLdegX(X)=10$ concluding the example.
\end{example}

In Example \ref{threeByThreeCase}, we used   Corollary \ref{MLdegformula} and topological computations to determine the ML degree of a singular variety. 
In the next section we will again use Corollary \ref{MLdegformula} and topological computations to determine ML degrees.

\section{The ML degree for rank $2$ matrices}\label{win}

To $\mathbb{P}^{mn}$, we associate the coordinates $p_{11},\dots,p_{1n},\dots,p_{m1},\dots,p_{mn},p_s$.
Let $\vX_{mn}$ denote the variety defined by 
$p_{11}+\cdots+p_{mn}-p_s=0$ and the vanishing of 
the $3\times 3$ minors of the matrix 
{\smaller
\begin{equation}\label{matrix}
\left[\begin{array}{cccc}
p_{11} & p_{12} & \dots & p_{1n}\\
\vdots & \vdots & \ddots & \vdots \\
p_{(m-1)1} & p_{(m-1)2} & \dots & p_{(m-1)n}\\
p_{m1} & p_{m2} & \dots & p_{mn}
\end{array}\right].
\end{equation}
}%
We  think of $\vX_{mn}$ as the Zariski  closure of the set of rank $2$ matrices in the distinguished hyperplane of $\mathbb{P}^{mn}$.
Let $Z_{mn}$ be the subvariety of $\vX_{mn}$ defined by the vanishing of the $2\times 2$ minors of the matrix (\ref{matrix}). Then $Z_{mn}$ is the singular locus of $X_{mn}$ for $m, n\geq 3$. 

Denote the very affine varieties associated to $X_{mn}$ and $Z_{mn}$ by $\oX_{mn}$ and $\oZ_{mn}$ respectively. 

We will make topological computations to determine the ML degree of $X_{mn}$. Summarizing, according to Lemma \ref{projectiveveryaffine}, $\MLdegX(X_{mn})=\MLdegY(\oX_{mn})$. Proposition \ref{Whitney} gives a Whitney decomposition of $X_{mn}$ and determines the Euler obstructions for this stratification. Since being Whitney regular and Euler obstruction are local invariants, we obtain a Whitney stratification of $\oX_{mn}$ and the Euler obstructions. 
Using these computations, we reduce our problem to determining a single Euler characteristic $\chi(\oX_{mn}\setminus \oZ_{mn})$ by Corollary \ref{substitute}.
Next,  Theorem \ref{TheoremPower} 
provides a closed form expression of 
$\chi(\oX_{mn}\setminus \oZ_{mn})$, for fixed $m$, in terms of the elements of a finite sequence  $\Lams_m$ of integers.
We conclude this section by computing $\Lambda_m$ for $m=3$, thereby proving 
Theorem~\ref{hrsWin}.


\subsection{Calculating Euler obstructions}\label{calcEO}

We will start by proving general results for  $m, n\geq 2$.
 At the end of this section, we will specialize to the case where $m=3$. 
 With some topological computations, we determine 
 $\lambda_1, \lambda_2$ of $\Lams_3$, thereby giving a closed form expression of the ML degree of $X_{3n}$ [Theorem \ref{hrsWin}].

To ease notation, we let $e_{(mn)}$ denote the Euler obstruction $e_{21}$ of the pair $\{Z_{mn}, X_{mn}\setminus Z_{mn}\}$, which is a Whitney stratification of $X_{mn}$ by Corollary \ref{WhitneyS}, and we have
$$\Ymn:=X_{mn}\setminus Z_{mn}.$$
\begin{prop}\label{Whitney}
Denote the Euler obstruction of the pair of strata $(Z_{mn},\Ymn)$ by $e_{(mn)}$. Then  
\begin{equation}\label{Eobstruction}
e_{(mn)}=(-1)^{m+n-1}(\min\{m, n\}-1).
\end{equation}
\end{prop}

\begin{proof}
Without loss of generality, we assume that $m\leq n$. When $m=2$, $X_{mn}$ is the distinguished hyperplane defined by 
$p_{11}+p_{12}+\cdots+p_{mn}=p_s$
and $Z_{mn}\subset X_{mn}$ is a smooth subvariety. Thus, the first part of the proposition follows by item {\ref{wsSmooth}} of Example \ref{WSexample}. Moreover, it follows from the definition of Euler obstruction that $e_{(mn)}=(-1)^{\dim Z_{mn}+1}$. Since $\dim Z_{mn}=m+n-2$, the second part of the proposition follows. 
Therefore, we can assume $m\geq 3$.

We will compute the Euler obstructions using Theorem \ref{Kashiwara}. 
Since $X_{mn}$ is contained in the distinguished hyperplane, consider $X_{mn}$ as a subvariety of $\pp^{mn-1}$, the projective space of all $m\times n$ matrices. 
To simplify notation, we will write $X$ and $Z$ instead of $X_{mn}$ and $Z_{mn}$. Then $S_1=X\setminus Z, S_2=Z$ form a Whitney stratification of $X$ according to Corollary \ref{WhitneyS}. 
Next, we will compute the Euler obstruction $e_{(mn)}$ of the pair $Z, X\setminus Z$ using Theorem \ref{Kashiwara}. 


One can easily compute that $\dim Z=m+n-2$. Under the notation of Theorem \ref{Kashiwara}, since $B\cap (X\setminus Z)\cap \phi^{-1}(\epsilon)$ is a complex manifold, it is orientable and even-dimensional, and hence
$$\chi_c(B\cap (X\setminus Z)\cap \phi^{-1}(\epsilon))=\chi(B\cap (X\setminus Z)\cap \phi^{-1}(\epsilon))$$
by Poincare duality. Since the Euler characteristic is homotopy invariant and since $\chi(\pp^{m-2})=m-1$, the second part of the proposition follows from the following lemma
where we show the link is homotopy equivalent to~$\pp^{m-2}$.
\end{proof}
\begin{lemma}\label{complexlink}
Suppose $m\leq n$. With the above notations, we have  $B\cap (X\setminus Z)\cap \phi^{-1}(\epsilon)$ is homotopy equivalent to $\pp^{m-2}$.
\end{lemma}
\begin{proof}
First, we give a concrete description of the normal slice $V$. Notice that $X\subset \mathbb{P}^{mn}$ is contained in the distinguished hyperplane $p_{11}+\cdots+p_{mn}-p_s=0$. 
In this proof, we will consider $X$ as a subvariety of $\mathbb{P}^{mn-1}$ with homogeneous coordinates $p_{11}, \ldots, p_{mn}$. 
Denote the affine chart $p_{11}\neq 0$ of $\mathbb{P}^{mn-1}$ by $U_{11}$. 
Let $a_{ij}=\frac{p_{ij}}{p_{11}}$ ($(i,j)\neq (1,1)$) be the affine coordinates of $U_{11}$ and let $a_{11}=1$. 
Denote the origin of $U_{11}$ by $O$. Let $O$ be the fixed point $z$ in Theorem \ref{Kashiwara}, and $U_{11}$ be the affine space $W$. 

Now, we define a projection $\pi: U_{11}\to Z\cap U_{11}$ by $[a_{ij}]\mapsto [b_{ij}]$, where $b_{ij}=a_{i1}\cdot a_{1j}$ and $a_{11}=1$. Then $U_{11}$ becomes a vector bundle over $Z\cap U_{11}$ via $\pi: U_{11}\to Z\cap U_{11}$. The preimage of $O$ is the vector space parametrized by $a_{ij}$ with $2\leq i\leq m$, $2\leq j\leq n$.

In terms of matrices, we can think of $\pi$ as the following map
{\smaller
$$
\left[\begin{array}{cccc}
1 & a_{12} & \cdots & a_{1n}\\
a_{21} & a_{22} & \cdots & a_{2n}\\
\vdots & \vdots & \ddots & \vdots\\
a_{m1} & a_{m2} & \cdots & a_{mn}
\end{array}\right]
\stackrel{\pi}{\to}
\left[\begin{array}{c}
1\\
a_{21}\\
\vdots\\
a_{m1}
\end{array}\right]\left[\begin{array}{c}
1\\
a_{12}\\
\vdots\\
a_{1n}
\end{array}\right]^{T}=\left[\begin{array}{cccc}
1 & a_{12} & \cdots & a_{1n}\\
a_{21} & a_{21}a_{12} & \cdots & a_{21}a_{1n}\\
\vdots & \vdots & \ddots & \vdots\\
a_{m1} & a_{m1}a_{12} & \cdots & a_{m1}a_{n}
\end{array}\right],$$
}
and we think of the preimage of $O$ as 
{\smaller
$$
\left[\begin{array}{cccc}
1 & 0 & \cdots & 0\\
0 & a_{22} & \cdots & a_{2n}\\
\vdots & \vdots & \ddots & \vdots\\
0 & a_{m2} & \cdots & a_{mn}
\end{array}\right]=\pi^{-1}\left(\left[\begin{array}{cccc}
1 & 0 & \cdots & 0\\
0 & 0 & \cdots & 0\\
\vdots & \vdots & \ddots & \vdots\\
0 & 0 & \cdots & 0
\end{array}\right]\right).
$$
}

By the above construction, we can take the normal slice $V$ at $O$ to be the fiber $\pi^{-1}(O)$. The intersection $V\cap X$ is clearly isomorphic to the affine variety $\{[a_{ij}]_{2\leq i\leq m, 2\leq j\leq n}|\text{rank}\leq 1\}$. Thus, we can define a map $\rho: V\cap (X\setminus Z)\to \pp^{m-2}$ which maps the matrix $\{[a_{ij}]_{2\leq i\leq m, 2\leq j\leq n}\}$ to one of its nonzero column vectors, as an element in $\pp^{m-2}$. Since the rank of $\{[a_{ij}]_{2\leq i\leq m, 2\leq j\leq n}\}$ is 1, the map does not depend on which nonzero column vector we choose. Using basic linear algebra, it is straightforward to check the following two statements about $\rho$.
\begin{itemize}
\item The restriction of $\rho$ to $B\cap (X\setminus Z)\cap \phi^{-1}(\epsilon)$ is surjective.
\item The restriction of $\rho$ to $B\cap (X\setminus Z)\cap \phi^{-1}(\epsilon)$ has convex fibers. 
\end{itemize}
Thus, $\rho$ is a fiber bundle with contractible fibers. By choosing a section, one can fiber-wise contract the total space of the fiber bundle to the chosen section.  Hence $\rho$ induces a homotopy equivalence between $B\cap (X\setminus Z)\cap \phi^{-1}(\epsilon)$ and $\pp^{m-2}$. For a more general and technical topological statement, see the main theorem of \cite{Smale}. 
\end{proof}

\begin{cor}\label{substitute}\label{Cor:EvsML}
Let $X_{mn}$ and $\oYmn$ be defined as in the beginning of this section. Then
\begin{equation}\label{Eq:EvsML}
\chi(\oYmn)=-\MLdegX(X_{mn})+(-1)^{m+n-1}(\min\{m, n\}-1). 
\end{equation}
\end{cor}
\begin{proof}
By Example \ref{topologicalExample}, we know that $\MLdegX(Z_{mn})=1$. One can easily compute that $\dim \oX_{mn}=2m+2n-3$. Now, the corollary follows from (\ref{Eobstruction}) and Corollary \ref{MLdegformula}.
\end{proof}

\subsection{Calculating Euler characteristics and ML degrees}\label{calcEC}
 In this subsection,  an expression for the Euler characteristic $\chi(\oYmn)$ is given to determine formulas for ML degrees. 

Recall that $\oX_{mn}$ is the complement of all the coordinate hyperplanes in $X_{mn}$, and that $Z_{mn}\subset X_{mn}$ is the subvariety corresponding to rank 1 matrices of size $m\times n$. 
\begin{thm}\label{TheoremPower}
Fix $m$ to be an integer greater than two. 
Then, there exists a sequence, denoted $\Lams_m$, of integers $\lambda_1, \lambda_2, \ldots, \lambda_{m-1}$ such that 
\begin{equation}\label{powerseries}
\chi(\oYmn)=
{(-1)^{n-1}}\sum_{1\leq i\leq m-1}\frac{\lambda_i}{i+1}
-\sum_{1\leq i\leq m-1}\frac{\lambda_i}{i+1}\cdot {i^{n-1}},
\quad\text{ for }n\geq 2.
\end{equation} 
\end{thm}

Before proving  Theorem \ref{powerseries}, 
we  quote a hyperplane arrangement result, which follows immediately from the theorem of Orlik-Solomon (see e.g. \cite[Theorem 5.90]{OT92}). 
\begin{lemma}\label{arrangement}
Let $L_1, \cdots, L_r$ be distinct hyperplanes in $\cc^s$. Suppose they are in general position, that is the intersection of any $t$ hyperplanes from $\{L_1, \cdots, L_r\}$ has codimension $t$, for any $1\leq t\leq s$. Denote the complement of $L_1\cup \cdots \cup L_r$ in $\cc^s$ by $M$. Then
\begin{itemize}
\item if $r=s+1$, then $\chi(M)=(-1)^{s}$;
\item if $r=s+2$, then $\chi(M)=(-1)^{s}(s+1)$.
\end{itemize}
\end{lemma}

{
We will use the proceeding lemma to compute Euler characteristics of stratum of the following stratifications. 
}
{
Throughout, we assume that $m$ is fixed and to simplify indices have $\mfY_n$ denoting $Y_{mn}^o$,
 the rank $2$ matrices with nonzero coordinates whose entries sum to $1$. 
 In other words,
$$\mfY_n=\left\{[a_{ij}]_{1\leq i\leq m, 1\leq j\leq n}\mathlarger{\mathlarger{\mathlarger{\mathlarger{|}}}}a_{ij}\in \cc^*, \;\sum_{1\leq i\leq m,1\leq j\leq n}a_{ij}=1,\, \rank[a_{ij}]=2\right\}.$$
A stratification of $\mfY_n$ is given by the number of columns summing to zero a matrix has. 
Defining $\mfY_n^{(l)}$ below, 
$$\mfY_n^{(l)}:=\left\{([a_{ij}]_{1\leq i\leq m, 1\leq j\leq n}\in \mfY_n\mathlarger{\mathlarger{\mathlarger{\mathlarger{|}}}} l= \#\left\{j\;{{\mathlarger{\mathlarger{|}}}}\;  \sum_{1\leq i\leq m}a_{ij}=0\right\}\right\},
$$
yields the stratification:
$$\mfY_n=\mfY_n^{(0)}\sqcup \cdots \sqcup \mfY_n^{(n-1)},$$
where each $\mfY_n^{(l)}$ is a locally closed subvariety of $\mfY_n$. 
Note that by definition,
$\mfY^{(0)}_n$ is the set of rank $2$ matrices with nonzero column sums, i.e., 
$$\mfY^{(0)}_n:=\left\{[a_{ij}]_{1\leq i\leq m, 1\leq j\leq n}\in \mfY_n\mathlarger{\mathlarger{\mathlarger{\mathlarger{|}}}}\sum_{1\leq i\leq m}a_{ij}\neq 0 \text{ for each } 1\leq j\leq n \right\}.$$
%
The following lemma will show $\chi(\mfY^{(0)}_n)=\chi (\mfY_n)$ by proving $\chi(\mfY_n^{(l)})=0$ for $l\geq 1$.
}

\begin{lemma}
The Euler characteristic of rank $2$ matrices that sum to one equals the Euler characteristic of rank $2$ matrices with nonzero column sums that sum to one. In other words, with notation as above, 
\begin{equation}\label{firsteq}
\chi\left(\mfY_n\right)=\chi\left(\mfY^{(0)}_n\right)\text{ and }\chi\left(\mfY_n^{(1)}\right)=\chi\left(\mfY_n^{(2)}\right)=\dots=\chi\left(\mfY_n^{(n-1)}\right)=0.
\end{equation}
\end{lemma}
\begin{proof}[Proof of Lemma]
We define a $\cc^*$ action on $\mfY_n$ by putting $t\cdot (a_{ij})=a'_{ij}$, where
\begin{equation*}
a'_{ij}=
\begin{cases}
a_{ij} & \text{if $a_{1j}+\cdots+a_{mj}\neq 0$}\\
t \times a_{ij} & \text{if $a_{1j}+\cdots+a_{mj}=0$}.
\end{cases}
\end{equation*}
It is straightforward to check the action preserves each $\mfY_n^{(l)}$.
The action is transitive and continuous on $\mfY_n^{(l)}$ for any $l\geq 1$.
Therefore, $\chi(\mfY_n^{(l)})=0$ for any $l\geq 1$, and hence 
$$\chi(\mfY_n)=\chi\left(\mfY_n^{(0)}\right).$$
\end{proof}

A column sum of a matrix in  $\mfY^{(0)}_n$ can be any element of $\mathbb{C}^*$. Now we consider a subset of $\mfY^{(0)}_n$ where the column sums are all exactly $1$.
 Let 
$V_n$ denote the set of rank $2$ matrices with column sums equal to $1$, i.e.,
$$V_n:=\left\{[b_{ij}]_{1\leq i\leq m, 1\leq j\leq n}\mathlarger{\mathlarger{\mathlarger{\mathlarger{|}}}}b_{ij}\in \cc^*, \;\sum_{1\leq i\leq m}b_{ij}=1 \text{ for each $1\leq j\leq n$, } \rank\left([b_{ij}]\right)=2\right\}.$$
 In Lemma \ref{lemmaUV}, we express the  Euler characteristic of $\mfY^{(0)}_n$ in terms of 
$\chi(V_n)$.

\begin{lemma}\label{lemmaUV}
\begin{equation}\label{UV}
\chi(\mfY^{(0)}_n)=(-1)^{n-1}\chi(V_n).
\end{equation}
\end{lemma}
\begin{proof}[Proof of Lemma]
Let $T_n=\{(t_j)_{1\leq j\leq n}\in (\cc^*)^n|\sum_j t_j=1\}$. 
Define a map $F: T_n\times V_n\to \mfY^{(0)}_n$ by putting $a_{ij}=t_j b_{ij}$; 
we think of the $j$th element of $T_n$ as scaling the $j$th column of $V_n$. 
Clearly, $F$ is an isomorphism. Therefore,
$\chi(\mfY^{(0)}_n)=\chi(T_n)\cdot \chi(V_n)$.
$T_n$ can be considered as $\cc^{n-1}$ removing $n$ hyperplanes in general position. By Lemma \ref{arrangement}, $\chi(T_n)=(-1)^{n-1}$, and hence $\chi(\mfY^{(0)}_n)=(-1)^{n-1}\chi(V_n)$. 
\end{proof}

A stratification for $V_n$ can be given by the minimal $j_0$ such that the column vector $[b_{ij_0}]$ is linearly independent from $[b_{i1}]$. Since $\sum_i b_{ij}=1$ for all $j$, if two column vectors are linearly dependent, they must be equal. 
Defining $V_n^{(l)}$ below,
 $$V_n^{(l)}:=\left\{[b_{ij}]\in V_n| \text{ the column vectors satisfy } [b_{i1}]=[b_{i2}]=\cdots=[b_{i(n-l+1)}]\neq [b_{i(n-l+2)}]\right\},$$
 yields the stratification of $V_n$ by  locally closed subvarieties:
$$
V_n=\bigsqcup_{l=2}^nV_n^{(l)}=V_n^{(2)}\sqcup V_n^{(3)}\sqcup \cdots \sqcup V_n^{(n)}.
$$
Hence
$$
\chi(V_n)=\chi\left(V_n^{(2)}\right)+\chi\left(V_n^{(3)}\right)+ \cdots +\chi\left(V_n^{(n)}\right). 
$$

We use $W_l$ to denote the $m\times l$ rank $2$ matrices with column sums equal to $1$ such that the first two column vectors are linearly independent. 
Then, we have the isomorphism $V_n^{(l)}\cong W_{l}$ by the following map,
$$
\begin{array}{ccccccccc}
[\underbrace{\begin{array}{cccc}
b_{i1} & b_{i1} & \dots & b_{i1}\end{array}} & b_{i\left(n-l+2\right)} & \dots & b_{in}] & \mapsto & [
\underbrace{\begin{array}{cccc}b_{i1} & b_{i(n-l+2)} & \dots & b_{in}\end{array}}].
\\
(n-l+1)\,\text{copies of }b_{i1} &&&&& l\,\text{ columns }
\end{array}$$
Therefore,  
\begin{equation}\label{Vsum}
\chi(V_n)=\chi(W_2)+\chi(W_3)+\cdots+\chi(W_n). 
\end{equation}

For any $l\geq 2$, we can define a map $\pi_l: W_l\to W_2$ by taking the first two column. Thus, we can consider all $W_l$ as varieties over $W_2$. The following lemma gives a topological description of $W_l$, which will be useful to compute its Euler characteristics. 
\begin{lemma}\label{power}
For any $l\geq 2$,
$$W_l\cong W_3\times_{W_2}W_3\times_{W_2}\cdots \times_{W_2} W_3$$
where there are $l-2$ copies of $W_3$ on the right hand side and the product is the topological fiber product. In other words, take any point $x\in W_2$ the fiber of $\pi_l: W_l\to W_2$ over $x$ is equal to the $(l-2)$-th power of the fiber of $\pi_3: W_3\to W_2$ over $x$.
\end{lemma}
\begin{proof}[Proof of Lemma]
Given $l-2$ elements in $W_3$. Suppose they all belong to the same fiber of $\pi_3: W_3\to W_2$. This means that we have $(l-2)$ size $m\times 3$ matrices of rank 2, which all have the same first two columns. Then we can collect the third column of each matrix, and put them after the same first two columns. Thus we obtain a $m\times l$ matrix, whose rank is still 2. In this way, we obtain a map $W_3\times_{W_2}W_3\times_{W_2}\cdots \times_{W_2} W_3\to W_l$, 
$$\begin{array}{ccc}
\underbrace{\begin{array}{cccc}
\left[\begin{array}{ccc}
b_{i1} & b_{i2} & b_{i3}\end{array}\right], & \left[\begin{array}{ccc}
b_{i1} & b_{i2} & b_{i4}\end{array}\right], & \dots, & \left[\begin{array}{ccc}
b_{i1} & b_{i2} & b_{il}\end{array}\right]\end{array}} & \mapsto & \left[\begin{array}{ccccc}
b_{i1} & b_{i2} & b_{i3} & \dots & b_{il}\end{array}\right],\\
\left(l-2\right)\,m\times3\,\text{matrices}
\end{array}
$$
which is clearly an isomorphism. 
\end{proof}
Given a point $x\in W_2$,  we let $F_x$ denote the fiber of $\pi_3: W_3\to W_2$ over $x$. 
Since Euler characteristic satisfies the product formula for fiber bundles, to compute the Euler characteristic of $W_l$, it suffices to study the stratification of $W_2$ by the Euler characteristic of $F_x$, and compute the Euler characteristic of each stratum. More precisely, let  
$$W_2^{(k)}:=\{x\in W_2| \chi(F_x)=-k\},
\text{ so }W_2=\bigsqcup_{k\in\mathbb{Z}}W_2^{(k)}.
$$ 
Then, to compute $\chi(W_l)$, 
 it suffices to compute $\chi(W_2^{(k)})$
 for all $k\in\mathbb{Z}$.
According to Lemma~\ref{boundChiFx}, it suffices to consider $k\in\{0,1,\dots,m-1\}$. For this reason,
 
let $\Lams_m$ denote the sequence $\lambda_0$, $\lambda_1,\dots,\lambda_{m-1}$, where $\lambda_k$ are defined by  
\begin{equation}\label{lambdaEq}
 \lambda_k:=\chi(W_2^{(k)}) \text{ for }0\leq k\leq m-1.
 \end{equation}

As we will soon see in the proofs, we can consider $F_x$ as the complement of some point arrangement in $\mathbb{C}$. The arrangement is parametrized by the point $x\in W_2$. In other words, $W_2$ is naturally a parameter space of point arrangement in $\mathbb{C}$. According to the Euler characteristic of the corresponding point arrangement, $W_2$ is canonically stratified. Our problem is to compute the Euler characteristic of each stratum. The main difficulty to generalize our method to compute the ML degree of higher rank matrices is to solve the corresponding problem for higher dimensional hyperplane arrangements. 

\begin{lemma}\label{boundChiFx}For any $x\in W_2$,
$$0\geq \chi(F_x)\geq -(m-1).$$
Moreover, the map $W_2\to \zz$ defined by $x\mapsto \chi(F_x)$ is a semi-continuous function. In other words, for any integer $k$, $\{x\in W_2|\chi(F_x)\geq -k\}$ is a closed algebraic subset of $W_2$.  In particular, the subsets $W_2^{(k)}=\{x\in W_2|\chi(F_x)= -k\}$, $0\leq  k\leq  (m-1)$, give a stratification of $W_2$ into locally closed subsets. 
\end{lemma}
\begin{proof}[Proof of Lemma]
By definition, 
$$W_2=\left\{[b_{ij}]_{1\leq i\leq m, j=1,2}\mathlarger{\mathlarger{\mathlarger{\mathlarger{|}}}}b_{ij}\in \cc^*, \sum_{1\leq i\leq m}b_{i1}=\sum_{1\leq i\leq m}b_{i2}=1, \rank\left([b_{ij}]\right)=2\right\}.$$
Fix an element $x=[b_{i1}\,b_{i2}]\in W_2$. By definition, the fiber $F_x$ of $\pi_3: W_3\to W_2$ is equal to the following:
%
$$F_x=\left\{[b_{i3}]\in(\mathbb{C}^*)^m\mathlarger{\mathlarger{\mathlarger{\mathlarger{|}}}} \sum_{1\leq i\leq m}b_{i3}=1, \text{$[b_{i3}]$ is contained in the linear span of $[b_{i1}]$ and $[b_{i2}]$}\right\}.$$

Since the columns $[b_{i1}],[b_{i2}],[b_{i3}]$ each sum to one, for any $[b_{i3}]\in F_x$ there exists $T\in \cc$ such that $[b_{i3}]=T\cdot [b_{i1}]+(1-T)\cdot [b_{i2}]$. 
For each $i$, this means  $b_{i3}\neq 0$ is equivalent to $b_{i2}+T\cdot (b_{i1}-b_{i2})\neq 0$. Therefore, for $x= [b_{i1}\,b_{i2}]$,
\begin{equation}
\begin{array}{lll}
F_x &\cong& \left\{T\in\cc\mathlarger{\mathlarger{\mathlarger{|}}}T\neq-\frac{b_{i2}}{b_{i1}-b_{i2}} \text{ if } b_{i1}-b_{i2}\neq 0 \text{ for }i=1,2,\dots,m
\right\}
\\
&\cong &\cc\setminus\left\{-\frac{b_{i2}}{b_{i1}-b_{i2}}\mathlarger{\mathlarger{\mathlarger{|}}}1\leq i\leq m \text{ such that } b_{i1}-b_{i2}\neq 0\right\}.
\end{array}
\end{equation}
Notice that $[b_{i1}]\neq [b_{i2}]$. Therefore, there has to be some $i$ such that $b_{i1}\neq b_{i2}$. Thus $F_x$ is isomorphic to $\cc$ minus some points of cardinality between $1$ and $m$, and hence the first part of the lemma follows. 

The condition that $\chi(F_x)\geq r$ is equivalent to the condition of some number of equalities $b_{i1}=b_{i2}$ and some number of overlaps among $\frac{b_{i2}}{b_{i1}-b_{i2}}$. Those conditions can be expressed by algebraic equations. Thus, the locus of $x$ such that $\chi(F_x)\geq r$ is a closed algebraic subset in $W_2$. 
\end{proof}
\begin{example}\label{locusBound}
When $m=3$, the Euler characteristic $\chi(F_x)$ is in $\{0, -1, -2\}$.
For example if 
$$x_0=\left[\begin{array}{cc}
1&1\\
2&3\\
-2&-3
\end{array}\right],
\quad x_1=\left[\begin{array}{cc}
2&2\\
1&-2\\
-2&1
\end{array}\right],
\quad x_2=\left[\begin{array}{cc}
4&5\\
-2&7\\
-1&-11
\end{array}\right]$$
then $\chi(F_{x_0})=0,\chi(F_{x_1})=-1,\chi(F_{x_2})=-2$.
This is because $F_{x_0}\cong\mathbb{C}\setminus{\{3\}}$,
$F_{x_1}\cong\mathbb{C}\setminus{\{-\frac{2}{3},\frac{1}{3}   \}}$, and
$F_{x_2}\cong\mathbb{C}\setminus{\{5,-\frac{7}{9},\frac{-11}{10} \}}$.
More generally, every matrix $x\in W_2$ has $\chi(F_x)\geq -2$. 
To have $x\in W_2$ such that 
$\chi(F_x)\geq -1$, at least one of the following six polynomials must vanish:
$$
\begin{array}{ccc}
	(b_{31}-b_{32}),\quad&
	(b_{21}-b_{22}),\quad&
(b_{11}-b_{12}),\quad\\
\det\left[\begin{array}{cc}
b_{11} & b_{12}\\
b_{21} & b_{22}
\end{array}\right], &
\det\left[\begin{array}{cc}
b_{11} & b_{12}\\
b_{31}& b_{32}
\end{array}\right],    &
\det\left[\begin{array}{cc}
b_{21} & b_{22}\\
b_{31} & b_{32}
\end{array}\right].
	\end{array}
$$
These conditions induce six irreducible components in $W_2$ whose points satisfy $\chi(F_x)\geq -1$.
To have $x\in W_2$ such that 
 $\chi(F_x)\geq 0$, two of the the six polynomials above must vanish. However, this does not induce $15=\binom{6}{2}$ components in $W_2$. This is because in twelve of these cases there are no points in $W_2$ that satisfy the conditions. Instead, there are only three components.

\end{example}

\begin{proof}[
Proof of Theorem \ref{TheoremPower}
]

Thus far, we have shown with 
\cref{firsteq,UV,Vsum},
the following relations:

\begin{equation}\label{new}
\chi(\mfY_n)=\chi(\mfY^{(0)}_n)=(-1)^{n-1}\chi(V_n)=(-1)^{n-1}\sum_{2\leq l\leq n}\chi(W_l).
\end{equation}
Recall that $W_2^{(k)}=\{x\in W_2| \chi(F_x)=-k\}$, and by Lemma \ref{boundChiFx} we have the stratification, 
\begin{equation}
W_2=W_2^{(0)}\sqcup W_2^{(1)}\sqcup \cdots \sqcup W_2^{(m-1)}.
\end{equation}
Moreover, $W_2^{(k)}$ are locally closed algebraic subsets of $W_2$. 
The projection $\pi_3: W_3\to W_2$ induces  a fiber bundle 
over each $W_2^{(k)}$, and by definition of $W_2^{(k)}$, the fiber has Euler characteristic $-k$. 
With  Lemma \ref{power}, we showed $W_l$ is isomorphic to an $(l-2)$ fiber product of $W_3$ over $W_2^{(k)}$.
Thus,  restricting  $\pi_l: W_l\to W_2$ to $W_2^{(k)}$, the induced  bundle's fiber has  Euler characteristic $(-k)^{l-2}$. Then,
\begin{equation}\label{equalsum}
\chi(\pi_l^{-1}(W_2^{(k)}))=\lambda_k\cdot (-k)^{l-2}. 
\end{equation}
Since $W_2^{(k)}$ are locally closed algebraic subsets of $W_2$, $\pi_l^{-1}(W_2^{(k)})$ are locally closed algebraic subsets of $W_l$. Therefore, the additivity of Euler characteristic implies the following,
\begin{equation}\label{Wsum}
\chi(W_l)=\sum_{0\leq k\leq m-1}\chi(\pi_l^{-1}(W_2^{(k)}))=\sum_{0\leq k\leq m-1}\lambda_k \cdot (-k)^{l-2}. 
\end{equation}
Here our convention is $0^0=1$.

By equations (\ref{new}), (\ref{Wsum}) and  Proposition \ref{lambda0}, we have the following equalities.
$$
\chi(\oYmn)
=(-1)^{n-1}\sum_{2\leq l\leq n}\left(\sum_{0\leq k\leq m-1}\lambda_k\cdot (-k)^{l-2}\right)=(-1)^{n-1}\sum_{0\leq k\leq m-1}\lambda_k\cdot \frac{1-(-k)^{n-1}}{1-(-k)}.
$$
The last line becomes the same as (\ref{powerseries}) by replacing  $k$ by $i$ and showing $\lambda_0=0$ in Proposition~\ref{lambda0}.
\end{proof}

\begin{prop}\label{lambda0}
Let $\lambda_0$ of $\Lambda_m$ be defined as above, then  $\lambda_0=0$. 
\end{prop}

We divert the proof of Proposition \ref{lambda0} to the end of the section. 
Now, we specify to the case $m=3$ for our first main result. 
\begin{thm}{\bf{[Main Result]}}\label{hrsWin}
The maximum likelihood degree of $X_{3n}$ is given by the following formula. 
\begin{equation}
\MLdegX(X_{3n})=2^{n+1}-6. 
\end{equation}
\end{thm}
\begin{proof}
In \cite{HKS05}, the ML degree of $X_{32}$ and ML degree of $X_{33}$ are determined to be $1$ and $10$ respectively. 
With this information it follows $\lambda_1+\lambda_2=0$ and $\lambda_2=12$ by Theorem $\ref{TheoremPower}$.
\end{proof}

The take away is that finitely many computations  can determine infinitely many ML degrees. 
Using these techniques we may be able to determine ML degrees of other varieties, such as symmetric matrices and Grassmannians, with a combination of applied algebraic geometry and topological arguments.

\begin{proof}[Proof of Proposition \ref{lambda0}]
Recall that $\lambda_0=\chi(W_2^{(0)})$. By definition, $W_2^{(0)}$ consists of those $[b_{i1}\,b_{i2}]$ in $W_2$ such that the cardinality of the set $\{b_{i1}/b_{i2}|1\leq i\leq m, b_{i1}\neq b_{i2}\}$ is equal to 1.

Notice that for $[b_{i1}\,b_{i2}]\in W_2^{(0)}$,
\begin{equation}
\sum_{\substack{1\leq i\leq m \\ b_{i1}\neq b_{i2}}}b_{i1}=\sum_{\substack{1\leq i\leq m \\ b_{i2}\neq b_{i2}}}b_{i1}=0.
\end{equation}
Therefore, we can define a $\cc^*$ action on the set of $m$ by 2 matrices $\{[b_{ij}]_{1\leq i\leq m, j=1,2}\}$ by setting $t\cdot (b_{ij})=(b'_{ij})$, where 
\begin{equation}
b'_{ij}=
\begin{cases}
b_{ij} & \text{if $b_{i1}=b_{i2}$}\\
t \times b_{ij} & \text{otherwise}.
\end{cases}
\end{equation}
Now, it is straightforward to check this $\cc^*$ action preserves $W_2^{(0)}$ and the action is transitive on $W_2^{(0)}$. Recall that when an algebraic variety $M$ admits a $\mathbb{C}^*$ action, $\chi(M)=\chi(M^{\cc^*})$, where $M^{\cc^*}$ is the fixed locus of the action (see \cite[Proposition 1.2]{EM99}). Therefore, $\chi(W_2^{(0)})=\chi(\emptyset)=0$. 
\end{proof}

\section{Recursions and closed form expressions}\label{broaderApplications}
In this section, we use Theorem \ref{TheoremPower} to give recursions for the 
Euler characteristic $\chi(\oYmn)$ and thus the ML degree of $X_{mn}$ by 
 $\eqref{Eq:EvsML}$.  
We break the recursions and give closed form expressions in Corollary~\ref{compExpressions}.

\subsection{The recurrence}
Recall that we have $\oYmn:=\oX_{mn}\setminus Z_{mn}^o$.
By Theorem \ref{TheoremPower}, giving a recursion for $\chi(\oYmn)$
is equivalent to giving a  recursion for 
$-\MLdegX(X_{mn})+(-1)^{m+n-1}(m-1)$, assuming $m\leq n$. 
The next theorem gives the recursion for $\chi(\oYmn)$.

\begin{thm}
Fix $m$ and let $-c_i$ be the coefficient of $t^{m-i}$ in
$(t+1)\prod_{r=1}^{m-1}(t-r).$
For $n>m,$ we have 
$$\chi(Y_{mn}^o)=c_1 \chi(Y_{m(n-1)}^o)   +c_2 \chi(Y_{m(n-2)}^o)+\cdots +c_{m} \chi(Y_{m(n-m)}^o).$$

\end{thm}
\begin{proof}
By Theorem \ref{TheoremPower}, we have 
$$
\chi(\oYmn)=
{(-1)^{n-1}}\sum_{1\leq i\leq m-1}\frac{\lambda_i}{i+1}
-\sum_{1\leq i\leq m-1}\frac{\lambda_i}{i+1}\cdot {i^{n-1}}
$$
for $n\geq 2$. 
Therefore $\chi(\oYmn)$ is an 
order $m$ linear homogeneous recurrence relation with constant coefficients. 
The coefficients of such a recurrence are described by the characteristic polynomial with  roots $-1,1,\dots,m-1$, i.e. 
~$t^m-c_1t^{m-1}-\cdots -c_m=(t+1)\prod_{r=1}^{m-1}(t-r)$.
\end{proof}

With these recurrences we determine the following table of ML degrees:

\begin{table}[h]\label{smallTable}
\centerline{
$
\begin{array}{cccccccccccccc}
n=    & m=2& m=3&m=4 & m=5 & m=6 & m=7 \\
 m:& 1&   10 &         191 &          6776 &             378477 & 30305766       \\
 m+1:& 1&   26 &         843 &        40924 &           2865245 & 274740990       \\
 m+2:& 1&   58 &       3119 &      212936 &         19177197 & 2244706374        \\
 m+3:& 1&  122 &    10587 &    1015564 &      118430045  & 17048729886        \\
 m+4:& 1&  250 &    34271 &    4586456 &      692277357  & 122818757286        \\
 m+5:& 1&  506 &  107883 &  19984444 &    3892815965  & 850742384190        \\
 m+6:& 1& 1018 &  333839 & 84986216 &  21284701677  & 5720543812614\\
\vdots&\vdots&\vdots&\vdots&\vdots&\vdots&\vdots
\end{array}
$
}
\caption{\footnotesize ML degrees of $X_{mn}$. \normalsize} \label{tableMLD}
\end{table}
Note that our methods are not limited to Table \ref{smallTable}. We give closed form formulas in the next section for infinite families of mixture models. 

\subsection{Closed form expressions}
In this subsection we provide additional closed form expressions.
Using  an inductive procedure (described in the proof of Corollary \ref{compExpressions}),
we    determine $\Lams_m$ for small $m$ that can be extended to arbitrary $m$.%
\footnote{See the appendix for an implementation.} 

\begin{cor}\label{compExpressions}
For fixed $m=2,3,\dots7$, the closed form expressions  for  $\MLdegX(X_{mn})$ with $m\leq n$ are below:
%
{\smaller
$$
\begin{array}{ccl}
\MLdegX(X_{2n}) & = & 
1
\\
\MLdegX(X_{3n}) & = & 
\left(
\frac{-12}{2} \cdot{1^{n-1}}+
\frac{12}{3}\cdot {2^{n-1}}\right)
\\
\MLdegX(X_{4n}) & = & 
\left(
\frac{50}{2} \cdot{1^{n-1}}+
\frac{-120}{3}\cdot {2^{n-1}}+
 \frac{72}{4}\cdot3^{n-1}  
\right)
\\
\MLdegX(X_{5n}) & = & 
\left(
\frac{-180}{2} \cdot{1^{n-1}}+
\frac{780}{3}\cdot {2^{n-1}}+
 \frac{-1080}{4}\cdot3^{n-1} +
 \frac{480}{5}\cdot4^{n-1}  
\right)
\\
\MLdegX(X_{6n}) & = & 
\left(
\frac{602}{2} \cdot{1^{n-1}}+
\frac{-4200}{3}\cdot {2^{n-1}}+
 \frac{10080}{4}\cdot3^{n-1}  +
 \frac{-10080}{5}\cdot4^{n-1}  +
 \frac{3600}{6}\cdot5^{n-1}  
\right)
\\
\MLdegX(X_{7n}) & = & 
\left(
\frac{-1932}{2} \cdot{1^{n-1}}+
\frac{-20412}{3}\cdot {2^{n-1}}+
 \frac{-75600}{4}\cdot3^{n-1}  +
 \frac{127680}{5}\cdot4^{n-1}  +
 \frac{-100800}{6}\cdot5^{n-1}  +
 \frac{30240}{7}\cdot6^{n-1}  
\right).
\end{array}
$$
}
\end{cor}

\begin{proof}
We find these formulas using an inductive procedure to determine $\Lams_m$ from $\Lams_{m-1}$.

With equations \eqref{Eq:EvsML} and \eqref{powerseries},
 we have the following $m-1$ relations with $n=2,3,\dots,m$:
{\tiny
$$\left[\begin{array}{c}
\MLdegX\left(X_{m2}\right)\\
\MLdegX\left(X_{m3}\right)\\
\vdots\\
\MLdegX\left(X_{mm}\right)
\end{array}\right]+\left(-1\right)^{m}\left[\begin{array}{c}
1\\
-2\\
\vdots\\
\left(-1\right)^{m}\left(m-1\right)
\end{array}\right]=
%
\mathbf M
\left[\begin{array}{c}
\lambda_{1}/2\\
\lambda_{2}/3\\
\vdots\\
\lambda_{m-1}/m
\end{array}\right],\text{ where } 
$$
$$\mathbf M:=\left(
\left[\begin{array}{cccc}
1^1 & 2^1 & \cdots & (m-1)^1\\
1^2 & 2^2 & \cdots & (m-1)^2\\
\vdots & \vdots & \ddots & \vdots\\
1^{m-1} & 2^{m-1} & \cdots & (m-1)^{m-1}
\end{array}\right]
-
\left[\begin{array}{cccc}
\left(-1\right)^{1} & \left(-1\right)^{1} & \cdots & \left(-1\right)^{1}\\
\left(-1\right)^{2} & \left(-1\right)^{2} & \cdots & \left(-1\right)^{2}\\
\vdots & \vdots &  & \vdots\\
\left(-1\right)^{m-1} & \left(-1\right)^{m-1} & \cdots & \left(-1\right)^{m-1}
\end{array}\right]
\right)
.$$
}%
For fixed $m$, this system of linear equations has $2m-2$ unknowns:   $\MLdegX (X_{mj})$ for $j=2,\dots,m$ and $\lambda_1,\dots,\lambda_{m-1}$ of $\Lams_m$.
By induction, we may assume we know $\Lams_{m-1}$. 
The $\Lams_{m-1}$ gives us a closed form expression for the ML degrees of $X_{(m-1)j}$ with $j\geq 2$.
Since  $\MLdegX(X_{(m-1)j})=\MLdegX(X_{j(m-1)})$, we have reduced our system of linear equations to $m$ unknowns by substitution. 
By Proposition \ref{lambdam}, we have  $\lambda_{m-1}$ of $\Lams_m$ equals $(m-1)\cdot m!$.
Substituting this value as well,  we  have a linear system of $m-1$ equations in $m-1$ unknowns: 
$\MLdegX(X_{mm}),\lambda_1,\lambda_2,\dots,\lambda_{m-2}$.
 
 A  linear algebra argument shows that there exists a unique solution of the system yielding each $\lambda_j$ of $\Lams_m$ as well as $\MLdegX(X_{mm})$. 
 It proceeds as follows.  Since the unknown $\MLdegX(X_{mm})$ appears only in the last equation, the linear system above has a unique solution if and only if the determinant of the $(m-2)\times(m-2)$ upper left submatrix of $\mathbf{M}$, denoted $\mathbf N$, 
 is nonsingular. 
%
Consider the row vector ${\bfc}:=[c_1,c_2,\dots,c_{m-2}]$, and 
let $f_{\bfc}(x)=c_1x^1+c_2x^2+\cdots+c_{m-2}x^{m-2}$.
 We will show that ${\bfc}$ is a null vector for $\mathbf N$ 
 if and only if $\bfc$ is the zero vector. 
If $\bfc$ is a null vector, 
then, by multiplying  ${\bfc}$ with  $\mathbf N$ 
we see 
$$[f_\bfc(1)-f_\bfc(-1),f_\bfc(2)-f_\bfc(-1),\dots,f_\bfc(m-2)-f_\bfc(-1)]^T=0.$$ 
In other words,  $(f_\bfc(x)-f_\bfc(-1))$ has $m-1$ distinct roots and must be the zero polynomial. 
This means $f_\bfc(x)$ is a constant, and $\bfc={\bf 0}$.

Using the inductive procedure described above, we determined the following table of $\Lams_m$ to  yield the closed form expressions we desired. 
{\smaller
\begin{table}[h]\label{LamTable}
\centerline{
$
\begin{array}{rrrrrrrrrrrr}
  &\lambda_1 &\lambda_2 &\lambda_3 &\lambda_4 &\lambda_5 &\lambda_6 \\
 \Lams_2:&   2\\
\Lams_3:&   -12 & 12\\
\Lams_4:&   50 & -120& 72\\
\Lams_5:&   -180 & 780 & -1080 & 480\\
\Lams_6:&   602 & -4200 & 10080 &-10080 &3600\\ 
\Lams_7:&   -1932&   20412&  -75600&  127680&  -100800&  30240
\end{array}$
}
\caption{\footnotesize The $\lambda_i$  of $\Lams_{m}$. \normalsize} 
\end{table}
}
\end{proof}

\begin{rem}\label{AminoAcidRemark}
With our recursive methods  we determined the ML degree of $X_{20,20}$ to be 
{\smaller$$\text{19 674 198 689 452 133 729 973 092 792 823 813 947 695}\approx1.967\times10^{40}.$$	}%
While it is currently infeasible to track this large number of paths via numerical homotopy continuation as  in 
\cite{HRS14},
it may be possible to track much fewer paths using the topology of the variety. 
Preliminary results by the \textit{2016 Mathematics Research Community in Algebraic Statistics}
\textit{Likelihood Geometry Group} for weighted independence models deform a variety with ML degree one to a variety with large ML degree. 
This deformation deforms the maximum likelihood estimate of one model to the estimate of another.
The stratifications presented here may lead to similar results in the case of mixtures.  
\end{rem}

\begin{prop}\label{lambdam}
Let $\lambda_k$ of $\Lams_m$ be defined as in \eqref{lambdaEq}. 
Then $\lambda_{m-1}$ of $\Lams_m$ equals $(m-1)\cdot m!$.
\end{prop} 
\begin{proof}
Recall that  $\lambda_{m-1}$ of $\Lams_m$ equals $\chi(W_2^{(m-1)})$. 
By definition, $W_2^{(m-1)}$ consists of all~$[b_{i1}\,b_{i2}]\in W_2$ such that $b_{i1}\neq b_{i2}$ for all $1\leq i\leq m$ and $b_{i1}/b_{i2}$ are distinct for $1\leq i\leq m$. 

Denote by $B_m$ the subset of $(\cc^*\setminus\{1\})^m$ corresponding to $m$ distinct numbers. 
Then, there is a natural map $\pi: W_2^{(m-1)}\to B_m$, defined by $[b_{ij}]\mapsto \left(\frac{b_{11}}{b_{12}}, \ldots, \frac{b_{m1}}{b_{m2}}\right)$. The map is surjective. Moreover, one can easily check that under the map $\pi$, $W_2^{(m-1)}$ is a fiber bundle over $B_m$. In addition, the fiber is isomorphic to the complement of $m$ hyperplanes in $\cc^{m-2}$ in general position. By Lemma \ref{arrangement}, the fiber has Euler characteristic $(-1)^{m-2}(m-1)$. 

The Euler characteristic of $B_m$ is equal to $(-1)^m\cdot m!$ by induction. In fact, $B_m$ is a fiber bundle over $B_{m-1}$ with fiber homeomorphic to $\cc^*\setminus \{m \text{ distinct points}\}$. Therefore, 
$$\chi(W_2^{(m-1)})=(-1)^{m-2}(m-1)\cdot (-1)^m m!=(m-1)\cdot m!.$$\end{proof}

\section{Conclusion and additional questions}

We have developed the topological tools to determine the ML degree of singular models. 
We proved a closed form expression for  $3 \times n$ matrices with rank $2$ conjectured by \cite{HRS14}. 
In addition, our results provide  a recursion to determine the ML degree of a mixture of independence models where the first random variable has $m$ states and the second random variable has $n$ states.
Furthermore, we have shown how a combination of computational algebra calculations and topological arguments can determine an infinite family of ML degrees. 

The next natural question is to determine the 
 ML degree for higher order mixtures (rank $r$ matrices for $r> 2$). 
Our results give closed form expressions in the corank $1$ case by \textit{maximum likelihood duality} \cite{DR14}. 
Maximum likelihood duality  is quite surprising here because our methods would initially suggest that the corank $1$ matrices have a much more complicated ML degree. 
it would be very nice to give a topological proof in terms of Euler characteristics of maximum likelihood duality. One possible approach is by applying these techniques to the \textit{dual maximum likelihood estimation problem} described in \cite{Rod14}.

Another question concerns the \textit{boundary components} of statistical models as described in \cite{KRS15} for higher order mixtures. 
Can we also use these topological methods to give closed form expressions of the ML degrees of the boundary components of the statistical~model?

Finally, one should notice that the formulas in Corollary \ref{compExpressions} for the ML degrees involve alternating signs. 
It would be great to give a canonical transformation of our alternating sum formula into a positive sum formula. 
One reason why this might be possible is motivated by the work of \cite{GR14}.
Here, entries of the data $u$ are degenerated to zero and some of the critical points go to the boundary of the algebraic torus (i.e. coordinates go to zero). 
Based on how the critical points go to the boundary, one partitions the ML degree into a sum of positive~integers.



\section{Acknowledgements}
We would like to  thank  Nero Budur and Jonathan Hauenstein  for introducing the authors to one another. 
We would also like to especially thank Jan Draisma for his comments to improve an earlier version of this manuscript. 
We thank Sam Evens for the explanation of Theorem \ref{Kashiwara}.

\bibliographystyle{abbrv}







\pagebreak
\section*{Appendix}

The following code creates a function called {\ttfamily recursiveMLDegrees}
 that will compute a table of ML degrees and other Euler characteristics. 
The input consists of two integers $m,n$ with $m\leq n$. 
The output is a sequence of four elements. 
\begin{itemize}
\item Element 0 of the output is the ML degree of $\Xmn.$
\item Element 1 of the output is the list $\Lambda_m={\lambda_1,\lambda_2,..,\lambda_{m-1}}$. 
\item Element 2 of the output is a table of ML degrees given by a list of lists. 
\item Element 3 of the output lists $\Lambda_i$ for $i=2,3,\dots,m$. 
\end{itemize}
For example, {\ttfamily recursiveMLDegrees(3,4)} returns {\smaller
\begin{verbatim}
(26, {-12, 12}, {{1, 1, 1}, {10, 26}}, {{2}, {-12, 12}})
\end{verbatim}}\noindent
and  {\ttfamily (recursiveMLDegrees(7,n))\_3} returns Table \ref{LamTable} for $n\geq 7$.
Here is the code. 
{\smaller
\begin{verbatim}
  ----MACAULAY2 CODE------------------------------------------------------------------
  recursiveMLDegrees=(M,N)->(mlR=QQ[a_0..a_N]**QQ[c_0..c_N]**QQ[ML_0..ML_N];
    ----mlEquations is a function that gives the relations in the proof of Corollary 4.2. 
    mlEquations=(m,n,mldegX)->( (-mldegX+(-1)^(m+n-1)*(min(m,n)-1))- 
        sum for i from 1 to m-1 list (a_i*( (-1)^(n-1)-i^(n-1))/(i+1)));
    ----M2N lists MLDegree(X_2,j) for j=2,3,...,N; These ML degrees are 1. 
    M2N:=for i from 2 to N list 1;
    ----tableML gives a list of lists of ML degrees.
    ----tableLam gives, for i=2,3,...,m, the lists \Lambda_i={\lamda_1,\lamda_2,..,\lamda_{i-1}.
    tableML:={M2N};tableLam:={{2}};
    ----The loop constructs \Lamda_i and ML degrees using the recursion in the proof of Cor. 4.2.
    for fixM from 3 to M do (
      mldList:=(for i from 2 to fixM-1 list tableML_(i-2)_(fixM-i))|{ML_fixM};
      solveI:=radical ideal({a_(fixM-1)-fixM!*(fixM-1)}|
        for fixN from 2 to fixM list   mlEquations(fixM,fixN,mldList_(fixN-2)));
      newLam:=for j from 1 to fixM-1 list a_j %solveI;
      newLamSub:=for j from 1 to fixM-1 list a_j=>newLam_(j-1);
      tableLam=append(tableLam,(toList newLamSub)/last);
      newMLdegrees:=for i from fixM to N list (sum for j from 1 to N list 
        if j>= fixM then 0 else sub(a_j/(j+1)*j^(i-1), a_j=>newLam_(j-1)));
      tableML=append(tableML,newMLdegrees));
    ----The output of the function has four elements. 
    return (last last tableML,last tableLam,    tableML,tableLam))
  ----EXAMPLE------------------------------------------------------------------
    INPUT: recursiveMLDegrees(3,4)
    OUTPUT: (26, {-12, 12}, {{1, 1, 1}, {10, 26}}, {{2}, {-12, 12}})
  ----EXAMPLE------------------------------------------------------------------
    INPUT: (recursiveMLDegrees(3,4))_0
    OUTPUT: 26
  ------------------------------------------------------------------------------
\end{verbatim}}

\end{document}